\documentclass{article}
\usepackage[T2A]{fontenc}
\usepackage[cp1251]{inputenc}
\usepackage{amssymb}
\usepackage{amsmath}
\usepackage{amsthm}

\title{\bf The relation of semiadjacency and semicompatibility in $\cap\,$-semigroups of transformations}
\author{W. A. Dudek and V. S. Trokhimenko}
\date{Communicated by M. V. Volkov}

\begin{document}
\sloppy \maketitle

\newtheorem{theorem}{Theorem}
\newtheorem{proposition}{Proposition}
\newtheorem{lemma}{Lemma}
\newtheorem{definition}{Definition}
\newtheorem{coll}{Corollary}
\newcommand{\pr}{\mbox{\rm pr}_1\,}
\newcommand{\prr}{\mbox{\rm pr}_2\,}
\newcommand{\spr}{\mbox{\scriptsize pr}_1\,}
\newcommand{\sprr}{\mbox{\scriptsize pr}_2\,}

\begin{abstract}\noindent
We consider semigroups of transformations (partial mappings
defined on a set $A$) closed under the set-theoretic intersection
of mappings treated as subsets of $A\times A$. On such semigroups
we define two relations: the relation of semicompatibility which
identifies two transformations at the intersection of their
domains and the relation of semiadjacency when the image of one
transformation is contained in the domain of the second. Abstract
characterizations of such semigroups are presented.\\

\end{abstract}

\medskip

{\bf 1}. Let $\mathcal{F}(A)$ be the set of all transformations (i.e., the partial maps) of a non-empty set $A$. The domain of $f\in\mathcal{F}(A)$ is denoted by $\pr f$, the image by $\prr f$. The symbol $\Delta_{\spr f}$ is reserved for the identity relation on $\pr f$. The composition (superposition) of maps $f,g\in\mathcal{F}(A)$ is defined as $(g\circ f)(a)=g(f(a))$, where for every $a\in A$ the left and right hand side are defined, or undefined, simultaneously (cf. \cite{Clif}). If the set $\Phi\subset\mathcal{F}(A)$ is closed with respect to such composition, then the algebra $(\Phi,\circ)$ is called a \textit{semigroup of transformations} (cf. \cite{Clif} or \cite{Sch3}). If $\Phi $ is also closed with respect to the set-theoretic intersection of transformations treated as subsets of $A\times A$, then the algebra $(\Phi,\circ,\cap)$ is called a \textit{$\cap$-semigroup of transformations}.

On such a $\cap$-semigroup we can consider the so-called \textit{semicompatibility relation} $\xi_{\Phi}$ defined as
follows:
\begin{equation}\label{f-38}
(f,g)\in\xi_{\Phi}\longleftrightarrow f\circ\Delta_{\spr g}=g\circ\Delta_{\spr f} .
\end{equation}
The algebraic system $(\Phi,\circ,\cap,\xi_{\Phi})$ is called a \textit{transformative $\cap$-semigroup of transformations}. The investigation of such semigroups was initiated by V.~V.~Vagner (cf. \cite{Vagner}) and continued by V. N. Sali\v{\i} and B. M. Schein (cf. \cite{49}, \cite{50} and \cite{Schein}). The first abstract characterization of a $\cap$-semigroup of transformations was found by V. S. Garvatski\v{\i} (cf. \cite{Garv}).

Some abstract characterizations of transformative $\cap$-semigroups of transformations can be deduced from results
proved in \cite{Dudtro2} and \cite{Trrel} for Menger $\cap$-algebras of $n$-place functions.

On $(\Phi,\circ)$ we can also consider the \textit{semiadjacency relation}
\[
\delta_{\Phi}=\{(f, g)\,|\,\prr f\subset\pr g\} .
\]

An abstract characterization of semigroups of transformations with this relation was given in \cite{pavl-2}. Later, in
\cite{garv-trokh}, was found an abstract characterization of semigroups of transformations containing these two relations,
i.e., an abstract characterization of an algebraic system $(\Phi,\circ,\xi_{\Phi},\delta_{\Phi})$. The $\cap$-semigroup of
transformations with the semiadjacency relation was described in \cite{Dud-Trokh}. The semiadjacency relation on algebras of
multiplace functions was investigated in \cite{trokhprim}. 

In this paper we find an abstract characterization of a $\cap$-semigroup of transformations equipped with the semicompatibility relation and the relation of semiadjacency.

We start with the following lemma.

\begin{lemma}\label{P-dt-1}
The relation of semiadjacency defined on a semigroup $(\Phi,\circ)$ satisfies the following two conditions:
\begin{eqnarray}
 & &\label{f-dt-1}(f,g)\in\delta_{\Phi}\longleftrightarrow\pr f\subset\pr (g\circ f),\\[4pt]
& & \label{f-dt-2} (f,g)\in\delta_{\Phi}\longrightarrow (f\circ h,g)\in\delta_{\Phi}.
\end{eqnarray}
\end{lemma}
We omit the proof of this lemma since it is a simple consequence of results proved in \cite{Dud-Trokh}, \cite{garv-trokh} and
\cite{pavl-2}.

\bigskip

{\bf 2}. Each homomorphism $P$ of an abstract semigroup $(G,\cdot)$ into the semigroup $(\mathcal{F}(A),\circ)$ of all
transformations of a set $A$ is called a \textit{representation of $(G,\cdot)$ by transformations}. In the case when a representation is an isomorphism we say that it is \textit{faithful}.

With each representation $P$ of a semigroup $(G,\cdot)$ by transformations of $A$ are associated three binary relations:
\begin{eqnarray}
 & &\zeta_P =\{(g_1,g_2)\,|\,P(g_1)\subset P(g_2)\},\nonumber\\[4pt]
 & &\xi_P =\{(g_1,g_2)\,|\,P(g_1)\circ\bigtriangleup_{\spr P(g_2)}=
 P(g_2)\circ\bigtriangleup_ {\spr P(g_1)}\},\nonumber\\[4pt]
 & &\delta_P=\{(g_1, g_2)\,|\,\prr P(g_1)\subset\pr P(g_2)\}\nonumber
\end{eqnarray}
defined on $G$.

Let $(P_i)_{i\in I}$ be a family of representations of a semigroup $(G,\cdot)$ by transformations of disjoint sets
$(A_i)_{i\in I}$, respectively. By the \textit{sum} of this family we mean the map $P\colon\,g\mapsto P(g)$, where $g\in G$, and
$P(g)$ is the transformation on $A=\bigcup\limits_{i\in I}A_i$ defined by
$$
P(g)=\bigcup\limits_{i\in I} P_i(g).
$$
Such defined $P$ is a representation of $(G,\cdot)$. It is denoted by $\sum\limits_{i\in I}P_i$.

If $P$ is the sum of the family of representations $(P_i)_{i\in I}$, then obviously 
\begin{equation} \label{f-dt-4}
\zeta_P=\bigcap\limits_{i\in I}\zeta_{P_i},\quad \xi_P
=\bigcap\limits_{i\in I}\xi_{P_i},\quad\delta_P
=\bigcap\limits_{i\in I}\delta_{P_i}.
\end{equation}

\medskip

{\bf 3}. Following \cite{Clif} and \cite{Sch3} a binary relation
$\rho$ on a semigroup $(G,\cdot)$ is called:
\begin{itemize}
 \item\textit{stable} or \textit{regular}, if for all $x,y,u,v\in G$
 \[
 (x,y)\in\rho\wedge (u,v)\in\rho\longrightarrow (xu,yv)\in\rho,
 \]
  \item\textit{left regular}, if for all $x,u,v\in G$
 \[
 (u,v)\in\rho\longrightarrow (xu,xv)\in\rho,
 \]
 \item\textit{right regular}, if for all $x,y,u\in G$
 \[
 (x,y)\in\rho\longrightarrow (xu,yu)\in\rho,
 \]
 \item {\it left ideal}, if for all $x,y,u\in G$
\[
(x,y)\in\rho\longrightarrow (ux,y)\in\rho,
\]
 \item \textit{right negative}, if for all $x,y,u\in G$
 \[
(x,yu)\in\rho\longrightarrow (x,y)\in\rho.
\]
\end{itemize}

A quasi-order $\rho$, i.e., a reflexive and transitive relation,
is stable if and only if it is left and right regular (cf.
\cite{Clif}, \cite{Sch3}). Similarly, it is right negative if
$(xy,x)\in\rho$ for all $x,y\in G$.

Let $(G,\cdot)$ be an arbitrary semigroup, $(G^*,\cdot)$ -- the semigroup obtained from
$(G,\cdot)$ by adjoining an identity $e\not\in G$.
By a \textit{determining pair} of a semigroup $(G,\cdot)$ we
mean an ordered pair $(\varepsilon,W)$, where $\varepsilon$ is a
right regular equivalence relation on the semigroup $(G^*,\cdot)$, and $W$ is the
empty set or an $\varepsilon$-class which is a right ideal of
$(G,\cdot)$. Let $(H_a)_{a\in A}$ be the collection of all
$\varepsilon$-classes (uniquely indexed by elements of $A$) such
that $H_a\ne W$. As is well known (cf. \cite{Sch3}) with each
determining pair $(\varepsilon,W)$ is associated the so-called
\textit{simplest representation} $P_{(\varepsilon,W)}$ of
$(G,\cdot)$ by trasformations defined in the following way:
\begin{equation} \label{f-dt-5}
  (a_1,a_2)\in P_{(\varepsilon,W)}(g)\longleftrightarrow
  H_{a_1}g\subset H_{a_2},
\end{equation}
where $g\in G$, $a_1, a_2\in A$.

From results proved in \cite{Schein} and \cite{Sch3} we can
deduce the following properties of simplest representations.

\begin{proposition}\label{P-dt-2}
Let $(\varepsilon, W)$ be the determining pair of a semigroup
$(G,\cdot)$. Then for all $g_1,g_2\in G$, $x\in G^*$ we have
\begin{eqnarray}
& &\label{f-dt-6} (g_1,g_2)\in\zeta_{P_{(\varepsilon,W)}}\longleftrightarrow (\forall x)
(xg_1\not\in W\longrightarrow xg_1\equiv xg_2 (\varepsilon)),\\[4pt]
& &\label{f-dt-7} (g_1,g_2)\in\xi_{P_{(\varepsilon,W)}}\longleftrightarrow (\forall x) (xg_1\not\in W\wedge
xg_2\not\in W\longrightarrow xg_1\equiv xg_2 (\varepsilon)),\\[4pt]
& &\label{f-dt-8}
(g_1,g_2)\in\delta_{P_{(\varepsilon,W)}}\longleftrightarrow (\forall
x) (xg_1\not\in W \longrightarrow xg_1g_2\not\in W).
\end{eqnarray}
\end{proposition}

\begin{proposition}\label{P-p2} If a semigroup $(G,\cdot)$ and a semilattice $(G,\curlywedge )$
satisfy the identity
\begin{equation}\label{f-43}
  x (y\curlywedge z) = xy\curlywedge xz,
\end{equation}
then
\begin{equation}\label{f-5}
P_{(\varepsilon, W)}(g_1\curlywedge g_2) = P_{(\varepsilon,
W)}(g_1)\cap P_{(\varepsilon, W)}(g_2)
\end{equation}
holds for arbitrary elements $g_1,g_2\in G$ and a determining pair
$(\varepsilon, W)$ of $(G,\cdot)$ if and only if
\begin{eqnarray}
& &\label{f-6} g_1\in W\longrightarrow g_1\curlywedge
g_2\in W,\\
& &\label{f-7} g_1\curlywedge g_2\not\in W\longrightarrow
g_1\equiv g_2 (\varepsilon),\\
& &\label{f-8} g_1\not\in W\wedge g_1\equiv g_2
(\varepsilon)\longrightarrow g_1\curlywedge g_2\equiv
g_1(\varepsilon).
\end{eqnarray}
\end{proposition}

An analogous result was proved in \cite{Trokh1} (see also
\cite{Dudtro2}) for Menger algebras of rank $n$. For $n=1$ it
gives the above proposition.

\bigskip

{\bf 4}. In this section we will consider a {\it semilattice algebraic system} $(G,\cdot,\curlywedge,\xi,\delta)$, i.e., an algebraic system $(G,\cdot,\curlywedge,\xi,\delta)$ such that $(G,\cdot)$ is a semigroup, $(G,\curlywedge)$ is
a semilattice, $\xi$ is a left regular binary relation on
$(G,\cdot)$ containing the natural order $\zeta$ of a semilattice
$(G,\curlywedge)$, $\delta$ is a left ideal relation on $(G,\cdot)$. Assume
that $(G,\cdot,\curlywedge,\xi,\delta)$ satisfies the condition
\eqref{f-43}, as well as the conditions:
\begin{eqnarray} & &\label{f-44} (x,y),(u,v)\in\zeta\,\wedge\,
(y,v)\in\xi\longrightarrow (u,x)\in\xi,\\
& &\label{f-45} (x,y)\in\xi\longrightarrow (x\curlywedge y) u =
xu\curlywedge yu,
\end{eqnarray}
where $x,y,z,u,v\in G$, \ $(x,y)\in\zeta\longleftrightarrow x\curlywedge y = x$. Moreover, we assume also
that in a semigroup $(G^*,\cdot)$ with the adjoined identity $e$ we
have $(e,e)\in\zeta$, $(e,e)\in\delta$ and $(x,e)\in\delta$ for all $x\in G$.

\begin{proposition}\label{P-TVS}
In a semilattice algebraic system $(G,\cdot,\curlywedge,\xi,\delta)$ the relation $\xi$ is
reflexive and symmetric; the relation $\zeta$ is stable on a
semigroup $(G,\cdot)$.
\end{proposition}
\begin{proof}
The relation $\xi$ is reflexive since $\zeta\subset\xi$ and
$\zeta$ is the natural order on the semillatice $(G,\curlywedge)$. It
also is symmetric because for any $(x,y)\in\xi$ we have
$x\zeta x$, $y\zeta y$ and $x\xi y$, whence, by
\eqref{f-44}, we obtain $(y,x)\in\xi$.

To prove that $\zeta$ is stable on the semigroup $(G,\cdot)$ assume
that $(x,y)\in\zeta$ for some $x,y\in G$. Then $x\curlywedge
y=x$. Hence $z(x\curlywedge y)=zx$, which, by \eqref{f-43}, gives
$zx\curlywedge zy = zx$. Thus $(zx,zy)\in\zeta$. So, $\zeta$ is
left regular. Since $\zeta\subset\xi$, from $(x,y)\in\zeta$ it
follows $(x,y)\in\xi$, which, by \eqref{f-45}, implies
$(x\curlywedge y)z=xz\curlywedge yz$. Hence $xz=xz\curlywedge yz$,
i.e., $(xz,yz)\in\zeta$. This means that $\zeta$ is right regular.
Consequently, $\zeta$ is stable on the semigroup $(G,\cdot)$.
\end{proof}

Further for the sake of simplicity the formula $x\delta y\,\wedge\,xy\zeta
z$ will be written as $x\boxdot y\zeta z$.

\begin{definition}\label{D-d3a}\rm A subset $H\subset G$ is {\it $f_{\xi}$-closed} if the implication
\begin{equation}\label{f-54}
(u,v)\in\xi\,\wedge\,(u\curlywedge v)x\boxdot y\zeta
zt\,\wedge\, u,vx\in H\longrightarrow z\in H
\end{equation}
is valid for all $x,y,t\in G^*$ and $z,u,v\in G$.
\end{definition}

Clearly the set of all $f_{\xi}$-closed subsets of $G$ forms a complete lattice of subsets of
$G$ under intersection. Define $f_{\xi}(X)$ to be the least $f_{\xi}$-closed subset of $G$ containing
$X\subset G$.

\begin{proposition}\label{P-p9a} A non-empty subset $H$ of a semilattice algebraic
system $(G,\cdot,\curlywedge,\xi,\delta)$ is $f_{\xi}$-closed if
and only if it satisfies the conditions
\begin{eqnarray}
& &\label{f-22} xy\in H\longrightarrow x\in H,\\
& &\label{f-23} (g_1,g_2)\in\delta\,\wedge\, g_1\in H\longrightarrow g_1g_2\in H,\\
& &\label{f-24} g_1\curlywedge g_2=g_1\in H\longrightarrow g_2\in H,\\
& &\label{f-55} (g_1,g_2)\in\xi\,\wedge\, g_1, g_2x\in
H\longrightarrow (g_1\curlywedge g_2)x\in H,
\end{eqnarray}
where $x$ in the condition $\eqref{f-55}$ may be the empty symbol.
\end{proposition}
\begin{proof} Let $H$ be an $f_{\xi}$-closed subset of $G$. Then
\begin{equation}\label{f-56}
(u,v)\in\xi\,\wedge\, (u\curlywedge v) x\delta y\,\wedge\,
(u\curlywedge v) xy\zeta zt\,\wedge\, u,vx\in H\longrightarrow
z\in H
\end{equation}
for all $x,y,t\in G^*$ and $z,u,v\in G$.

Using $\eqref{f-56}$ we can prove conditions
$\eqref{f-22}-\eqref{f-55}$. Indeed, for $u =v=xy$, $x=y=e$,
$t=y$, $z=x$ the implication $\eqref{f-56}$ has the form
\[ (xy,xy)\in\xi\,\wedge\, (xy\curlywedge
xy)e\delta e\,\wedge\, (xy\curlywedge xy)e\zeta xy\,\wedge\,
xy,xye\in H\longrightarrow x\in H.
\]
Since relations $\xi$ and $\zeta$ are reflexive and the operation
$\curlywedge$ is idempotent, the last condition is equivalent to
the implication $\eqref{f-22}$.

For  $u=v=g_1$, $x=e$, $y=g_1$, $t=e$, $z=g_1g_2$ the implication
$\eqref{f-56}$ gives the condition
\[ (g_1,g_1)\in\xi\,\wedge\,
(g_1\curlywedge g_1)e\delta g_2\,\wedge\,(g_1\curlywedge
g_1)eg_2\zeta g_1g_2e\,\wedge\,g_1, g_1e\in H\longrightarrow
g_1g_2\in H,
\]
which is equivalent to $\eqref{f-23}$.

Similarly for $u=v=g_1$, $x=y=t=e$, $z=g_2 $ from $\eqref{f-56}$
we obtain
\[
(g_1,g_1)\in\xi\,\wedge\, (g_1\curlywedge g_1)e\delta
e\,\wedge\,(g_1\curlywedge g_1)ee\zeta g_2e\,\wedge\,g_1,
g_1e\in H \longrightarrow g_2\in H,
\]
i.e., $(g_1,g_2)\in\zeta \wedge g_1\in H\longrightarrow g_2\in H$.
So, $\eqref{f-56}$ implies $\eqref{f-24}$.

Finally, $\eqref{f-56}$ for $u=g_1$, $v=g_2$, $y=e$,
$z=(g_1\curlywedge g_2)x$, $t=e$, gives
\[
(g_1,g_2)\in\xi \wedge (g_1\curlywedge g_2)x\delta e\wedge
(g_1\curlywedge g_2)xe\zeta (g_1\curlywedge g_2)xe\wedge g_1,
g_2x\in H\longrightarrow (g_1\curlywedge g_2)x\in H,
\]
which implies $\eqref{f-55}$.

To prove the converse assume that the conditions $\eqref{f-22}$
$\eqref{f-23}$, $\eqref{f-24}$, $\eqref{f-55}$ and the premise of
$\eqref{f-56}$ are satisfied. Then from $(u,v)\in\xi\wedge u,
vx\in H$, according to $\eqref{f-55}$, we obtain $(u\curlywedge
v)x\in H$. Since $(u\curlywedge v)x\delta y$, by $\eqref{f-23}$,
the last condition implies $(u\curlywedge v)xy\in H$. But
$(u\curlywedge v)xy\zeta zt$, by $\eqref{f-24}$, gives $zt\in
H$, which by $\eqref{f-22}$ gives $z\in H$. Thus, $\eqref{f-22}$,
$\eqref{f-23}$, $\eqref{f-24}$, $\eqref{f-55}$ imply
$\eqref{f-56}$.
\end{proof}

For a non-empty subset $H$ of $G$ we define the set
\[
F_{\xi}(H)=\{z\,|\,(\exists u,v,x,y,t) \;(u,v)\in\xi\wedge
(u\curlywedge v)x\boxdot y\zeta zt\wedge u,vx\in H)\},
\]
where $x,y,t\in G^*$ and $z,u,v\in G$. 

\begin{lemma}\label{Lem} For any subsets $H,$ $H_1,$ $H_2$ of $G$ we have

$(a)$ \ $H\subset F_{\xi}(H)$,

$(b)$ \ $F_{\xi}(H_1)\subset F_{\xi}(H_2)$\ for $H_1\subset H_2$.

$(c)$ \ $F_{\xi}(H)=H$ \ for any $f_{\xi}$-closed subset $H$ of
$G$.
\end{lemma}
\begin{proof}
Indeed, if $z\in H$, then
$$
(z,z)\in\xi\wedge(z\curlywedge z)e\boxdot e\zeta ze\wedge z,
ze\in H ,
$$
which means that $z\in F_{\xi}(H)$. Hence, $H\subset F_{\xi}(H)$.

The second condition is obvious.

To prove the last condition assume that $H$ is an $f_{\xi}$-closed
subset of $G$. Then for any $z\in F_{\xi}(H)$ and some $x,y,t\in
G^*$, $u,v\in G$ we have
$$
(u,v)\in\xi\wedge(u\curlywedge v)x\boxdot y\zeta zt\wedge
u,vx\in H.
$$
Since $H$ is $f_{\xi}$-closed, the above implies $z\in H$. Thus
$F_{\xi}(H)\subset H,$ which together with $(a)$ proves
$F_{\xi}(H)=H$.
\end{proof}

Further, for a non-empty subset $H$ of $G$ we put \ $\stackrel{0}{F}_{\xi}\!\!(H)=H$ and $\stackrel{n}{F}_{\xi}\!\!(H)=
F_{\xi}\Big(\!\!\stackrel{n-1}{F}_{\!\!\!\xi}\!\!(H)\Big)$ for any positive integer $n$. Then, by Lemma \ref{Lem}, we have
$$
H=\stackrel{0}{F}_{\xi}\!\!(H)\subset\stackrel{1}{F}_{\xi}\!\!(H)\subset\stackrel{2}{F}_{\xi}\!\!(H)\subset\stackrel{3}{F}_{\xi}\!\!(H)\subset\ldots$$

\begin{proposition}\label{P-10a}
For any subset $H$ of a semilatice algebraic system
$(G,\cdot,\curlywedge,\xi,\delta)$ we have
\begin{equation}\label{f-57}
 f_{\xi}(H)=\bigcup\limits_{n=0}^{\infty}\!\stackrel{n}{F}_{\xi}\!(H).
\end{equation}
\end{proposition}
\begin{proof}
Let
$\overline{H}_{\xi}=\bigcup\limits_{n=0}^{\infty}\!\stackrel{n}{F}_{\xi}\!(H)$
and
$$
(u,v)\in\xi\wedge (u\curlywedge v)x\delta y\wedge (u\curlywedge
v)xy\zeta zt\wedge u, vx\in\overline{H}_{\xi}\, ,
$$
for some $x,y,t\in G^*$ and $z,u,v\in G$. Since $u,
vx\in\overline{H}_{\xi}$, there are natural numbers $n_1,n_2$ such
that $u\in\stackrel{\;n_1}{F}_{\!\!\xi}\!\!(H)$ and
$vx\in\stackrel{\;n_2}{F}_{\!\!\xi}\!\!(H)$. Hence
$\stackrel{\;n_i}{F}_{\!\!\xi}\!(H)\subset\stackrel{n}{F}_{\xi}\!(H)$,
$i=1,2$, for $n=\max(n_1, n_2)$. Therefore
$$
(u,v)\in\xi\wedge (u\curlywedge v) x\boxdot y\zeta zt\wedge
u, vx\in\stackrel{n}{F}_{\xi}\!(H)\,  ,
$$
so,
$z\in\stackrel{n+1}{F}_{\!\!\xi}\!\!(H)\subset\overline{H}_{\xi}$.
This proves that $\overline{H}_{\xi}$ is a $f_{\xi}$-closed subset
of $G$.

By the definition $H\subset f_{\xi}(H)$. Hence, by Lemma
\ref{Lem}, ${F}_{\xi}(H)\subset F_{\xi}(f_{\xi}(H)) = f_{\xi}(H)$.
Similarly, $\stackrel{2}{F}_{\xi}\!\!(H)\subset f_{\xi}(H)$, etc.
Consequently, $\stackrel{\,n}{F}_{\xi}\!\!(H)\subset f_{\xi}(H)$
for any $n$, which implies
$\bigcup\limits_{n=0}^{\infty}\!\!\stackrel{n}{F}_{\xi}\!\!(H)\subset
f_{\xi}(H)$, i.e., $\overline{H}_{\xi}\subset f_{\xi}(H)$. On the
other hand,
$H\subset\bigcup\limits_{n=0}^{\infty}\!\!\stackrel{n}{F}_{\xi}\!\!(H)
=\overline{H}_{\xi}$. Therefore $f_{\xi}(H)\subset
f_{\xi}(\overline{H}_{\xi}) =\overline{H}_{\xi}$. Thus
$\overline{H}_{\xi}=f_{\xi}(H)$, which completes the proof of
\eqref{f-57}.
\end{proof}

Using the method of mathematical induction we can easily prove the
following proposition:

\begin{proposition}\label{P-10ab} For each
subset $H$ of a semilattice algebraic system
$(G,\cdot,\curlywedge,\xi,\delta)$, any natural number $n>1$ and
any $z\in G$ we have $z\in\stackrel{n}{F}_{\xi}\!\!(H)$ if and
only if following system of conditions
$$
\left(\begin{array}{c}
(u_1,v_1)\in\xi\wedge (u_1\curlywedge v_1) x_1\boxdot y_1\zeta zt_1, \ \ \ \ \ \ \ \ \ \\
\bigwedge\limits_{i=1}^{2^{n-1}-1}\left(\begin{array}{l}
(u_{2i},v_{2i})\in\xi\wedge (u_{2i}\curlywedge v_{2i})x_{2i}\boxdot y_{2i}\zeta
u_it_{2i},\\
 (u_{2i +1},v_{2i+1})\in\xi\wedge (u_{2i +1}\curlywedge v_{2i +1})x_{2i+1}\boxdot
y_{2i+1}\zeta v_ix_it_{2i+1}
\end{array}\right),\\
\bigwedge\limits_{i=2^{n-1}}^{2^n-1}(u_i, v_ix_i\in H)
 \end{array}\right)
$$
is valid for some\ $x_i,y_i,t_i\in G^*$ and $u_i,v_i\in G$.
\end{proposition}

In the sequel the system of the above conditions will be denoted
by $\frak{X}_n(z,H)$.

\bigskip

{\bf 5}. Let $(\Phi,\circ,\cap,\xi_{\Phi},\delta_{\Phi})$ be a
transformative $\cap$-semigroup of transformations with the
relation of semicompatibility $\xi_{\Phi}$ and the relation of
semiadjacency $\delta_{\Phi}$.

\begin{proposition}\label{P-11a}
$\bigcap\limits_{\varphi_i\in
H_{\Phi}}\!\!\!\pr\varphi_i\subset\pr\varphi$ \ for every
$H_{\Phi}\subset\Phi$ and $\varphi\in f_{\xi_{\Phi}}(H_{\Phi})$.
\end{proposition}
\begin{proof}
First we show that the following implication
\begin{equation} \label{f-59}
 \varphi\in\stackrel{n}{F}_{\xi_{\Phi}}\!\!(H_{\Phi})\longrightarrow\bigcap
 \limits_{\varphi_i\in H_{\Phi}}\!\!\pr\varphi_i\subset\pr\varphi
\end{equation}
is valid for every integer $n$. We prove it by induction.

 Let $\,\frak{A}=\!\!\!\!\bigcap\limits_{\varphi_i\in H_{\Phi}}\!\!\pr\varphi_i$. If $n=0$ and
$\varphi\in\stackrel{0}{F}_{\xi_{\Phi}}\!(H_{\Phi})$, then clearly
$\varphi\in H_{\Phi}$. Thus $\frak{A}\subset\pr\varphi$, which
verifies \eqref{f-59} for $n=0$.

Assume now that \eqref{f-59} is valid for some $n>0$. To prove
that it is valid for $n+1$ consider an arbitrary transformation
$\varphi\in\stackrel{n+1}{F}\!\!\!_{\xi_{\Phi}} (H_{\Phi})$. Then,
for some transformations $x,y,t,u,v\in\Phi$, where $x,y,t$ may be
the empty symbols, we have $(u,v)\in\xi_{\Phi}$, $\,(x\circ (u\cap
v),y)\in\delta_{\Phi}$, $\,y\circ x\circ (u\cap v)\subset
t\circ\varphi$ and $u,x\circ v\in\stackrel{n}{F}_{\xi_{\Phi}}\!\!
(H_{\Phi})$. The last condition, according to the assumption on
$n$, implies $\frak{A}\subset\pr u$. Similarly, $\frak{A}\subset\pr
(x\circ v)\subset\pr v$. Consequently $\triangle_{\spr
u}\circ\triangle_{\frak{A}}=\triangle_{\frak{A}}$ and
$\triangle_{\spr v}\circ\triangle_{\frak{A}}=
\triangle_{\frak{A}}$.

From $(x\circ (u\cap v),y)\in\delta_{\Phi}$ it follows $\prr
(x\circ (u\cap v))\subset\pr y$, which, by \eqref{f-dt-1}, gives
$\pr (x\circ(u\cap v))\subset\pr (y\circ x\circ(u\cap
v))\subset\pr (t\circ\varphi)$. Then, $(u,v)\in\xi_{\Phi}$ means
that $u\circ\triangle_{\spr v} = v\circ\triangle_{\spr u}$. So,
$u\circ\triangle_{\spr v}\circ\triangle_{\frak{A}}=
v\circ\triangle_{\spr u}\circ\triangle_{\frak{A}}$, hence
$u\circ\triangle_{\frak{A}}=v\circ\triangle_{\frak{A}}=
u\circ\triangle_{\frak{A}}\cap v\circ\triangle_{\frak{A}}=(u\cap
v)\circ\triangle_{\frak{A}}$. Since $\frak{A}\subset\pr (x\circ
v)$, we have
\[\arraycolsep=.5mm
\begin{array}{rll}
 \frak{A}&\subset\pr (x\circ v\circ\triangle_{\frak{A}})=\pr (x\circ(u\cap v)\circ
 \triangle_{\frak{A}})\subset \pr (y\circ x\circ (u\cap v)\circ\triangle_{\frak{A}})\nonumber\\[4pt]
 & \subset\pr (t\circ\varphi\circ\triangle_{\frak{A}})\subset\pr(\varphi\circ
 \triangle_{\frak{A}})
 =\pr(\varphi\circ\triangle_{\spr\varphi}\circ\triangle_{\frak{A}})
 \nonumber\\[4pt]
&=\pr(\varphi\circ\triangle_{\frak{A}}\circ\triangle_{\spr\varphi})\subset\pr\varphi.\nonumber
\end{array}\]
Thus, $\frak{A}\subset\pr\varphi$. This shows that \eqref{f-59} is
valid for $n+1$. Consequently, \eqref{f-59} is valid for all
integers $n$.

To complete the proof of this proposition observe now that,
according to \eqref{f-57}, for every $\varphi\in
f_{\xi_{\Phi}}(H_{\Phi})$ there exists $n$ such that
$\varphi\in\stackrel{n}{F}_{\xi_{\Phi}}\!(H_{\Phi})$, which, by
\eqref{f-59}, gives $\bigcap\limits_{\varphi_i\in
H_{\Phi}}\!\!\pr\varphi_i\subset\pr\varphi$.
\end{proof}

\begin{theorem}\label{T-6a}
An algebraic system $(G,\cdot,\curlywedge,\xi,\delta)$, where
$(G,\cdot)$ is a semigroup, $(G,\curlywedge)$ is a semilattice,
$\xi,\delta$ are binary relations on $G$, is isomorphic to some
transformative $\cap$-semigroup of transformations
$(\Phi,\circ,\cap,\xi_{\Phi},\delta_{\Phi})$ if and only if
$\,\xi$ is a left regular relation containing a semilattice order
$\zeta$, $\delta$ is a left ideal relation on $(G,\cdot)$, the conditions
$\eqref{f-43}$, $\eqref{f-44}$, $\eqref{f-45}$, as well as the
conditions:
\begin{eqnarray}
& &\label{f-65} x\curlywedge y\in f_{\xi}(\{x\})\longrightarrow
x\zeta y,\\[4pt]
&&\label{f-66} x\curlywedge y\in f_{\xi}(\{x,y\})\longrightarrow
x\xi y,\\[4pt]
&&\label{f-67} xy\in f_{\xi}(\{x\})\longrightarrow x\delta y
\end{eqnarray}
are satisfied by all elements of $G$.
\end{theorem}
\begin{proof} {\sc Necessity}.
Let $(\Phi,\circ,\cap,\xi_{\Phi},\delta_{\Phi})$ be a transformative
$\cap$-semigroup of transformations of some set. We show that it
satisfies all the conditions of our theorem.

The necessity of $\eqref{f-43}$ is a consequence of results proved
in \cite{Clif} and \cite{Garv}. Since the order $\zeta_{\Phi}$ of
a semilattice $(\Phi,\cap)$ coincides with the inclusion,
$\zeta_{\Phi}$ is contained in $\xi_{\Phi}$. From \eqref{f-dt-2}
(Lemma \ref{P-dt-1}) it follows that $\delta_{\Phi}$ is a left
ideal relation.

Let $(f,g)\in\xi_{\Phi}$, i.e., $f\circ\Delta_{\spr g}=
g\circ\Delta_{\spr f}$. Then $f\circ\Delta_{\spr g}\circ h =
g\circ\Delta_{\spr f}\circ h$. Since $\Delta_{\spr g}\circ h =
h\circ\Delta_{\spr g\circ h}$ and $\Delta_{\spr f}\circ h =
h\circ\Delta_{\spr f\circ h}$, we have $f\circ h\circ\Delta_{\spr
g\circ h} = g\circ h\circ\Delta_{\spr f\circ h}$, which proves
$(f\circ h, g\circ h)\in\xi_{\Phi}$. Thus, $\xi_{\Phi}$ is left
regular.

If $f\subset g$, $h\subset p$ and $(g,p)\in\xi_{\Phi}$ for some
$f,g,h,p\in\Phi$, then $f=g\circ\Delta_{\spr f}$,
$h=p\circ\Delta_{\spr h}$ and $g\circ\Delta_{\spr p} =
p\circ\Delta_{\spr g}$. The last equality implies
$g\circ\Delta_{\spr p}\circ\Delta_{\spr f}\circ\Delta_{\spr h} =
p\circ\Delta_{\spr g}\circ\Delta_{\spr f}\circ\Delta_{\spr h}$.
Thus, $p\circ\Delta_{\spr h}\circ\Delta_{\spr g}\circ\Delta_{\spr
f} = g\circ\Delta_{\spr f}\circ\Delta_{\spr p}\circ\Delta_{\spr
h}$. Consequently, $h\circ\Delta_{\spr g}\circ\Delta_{\spr f} =
f\circ\Delta_{\spr p}\circ\Delta_{\spr h}$, which in view of
$\Delta_{\spr g}\circ\Delta_{\spr f} =\Delta_{\spr f}$ and
$\Delta_{\spr p}\circ\Delta_{\spr h}=\Delta_{\spr h}$ gives
$h\circ\Delta_{\spr f} = f\circ\Delta_{\spr h}$. Therefore, $(h,
f)\in\xi_{\Phi}$. So, \eqref{f-44} is satisfied.

To prove \eqref{f-45} let $(f,g)\in\xi_{\Phi}$, i.e.,
$f\circ\Delta_{\spr g}=g\circ\Delta_{\spr f}$. Since $f\cap g
=(f\cap g)\circ\Delta_{\spr g} = f\circ\Delta_{\spr g}\cap g =
g\circ \Delta_{\spr f}\cap g = g\circ\Delta_{\spr f}\,(=
f\circ\Delta_{\spr g}),$ we have $h\circ (f\cap g) = h\circ
f\circ\Delta_{\spr g}\cap h\circ g\circ\Delta_{\spr f} = (h\circ
f\cap h\circ g)\circ\Delta_{\spr g}\circ\Delta_{\spr f} = h\circ
f\circ\Delta_{\spr f}\cap h\circ g\circ\Delta_{\spr g} = h\circ
f\cap h\circ g.$ Thus $h\circ (f\cap g)=h\circ f\cap h\circ g$,
which proves \eqref{f-45}.

Now let $\varphi\cap\psi\in f_{\xi_{\Phi}}(\{\varphi\})$ for some
$\varphi,\psi\in\Phi$. Then
$\pr\varphi\subset\pr(\varphi\cap\psi)$, by Proposition
\ref{P-11a}. Hence $\pr (\varphi\cap\psi) =\pr\varphi$ since
$\pr(\varphi\cap\psi)\subset\pr\varphi$. Thus
$\varphi=\varphi\circ\bigtriangleup_{\spr\varphi}
=\varphi\circ\bigtriangleup_{\spr(\varphi\cap\psi)}
=\varphi\cap\psi\subset\psi$. This proves \eqref{f-65}, because
the inclusion $\subset$ coincides with the order $\zeta_{\Phi}$ of
the semilattice $(\Phi,\cap)$.

If $\varphi\cap\psi\in f_{\xi_{\Phi}}(\{\varphi,\psi\})$, then, by
Proposition \ref{P-11a},
$\pr\varphi\cap\pr\psi\subset\pr(\varphi\cap\psi)$, which together
with the obvious inclusion
$\pr(\varphi\cap\psi)\subset\pr\varphi\cap\pr\psi$ gives
$\pr(\varphi\cap\psi) =\pr\varphi\cap\pr\psi$. So,
$\varphi\circ\bigtriangleup_{\spr\psi}
=\varphi\circ\bigtriangleup_{\spr\varphi}\circ\bigtriangleup_{\spr\psi}
=\varphi\circ\bigtriangleup_{\spr\varphi\cap\,\spr\psi}
=\varphi\circ\bigtriangleup_{\spr(\varphi\cap \,\psi)}
=\varphi\cap\psi =\psi\circ\bigtriangleup_{\spr(\varphi\cap
\,\psi)} =\psi\circ\bigtriangleup_{\spr\psi\cap \,\spr\varphi}
=\psi\circ\bigtriangleup_{\spr\psi}\circ\bigtriangleup_{\spr\varphi}
=\psi\circ\bigtriangleup_{\spr\varphi}$. Thus
$\varphi\circ\bigtriangleup_{\spr\psi}
=\psi\circ\bigtriangleup_{\spr\varphi}$, i.e.,
$(\varphi,\psi)\in\xi_{\Phi}$. This proves \eqref{f-66}.

To prove the last condition let $\psi\circ\varphi\in
f_{\xi_{\Phi}}(\{\varphi\})$. Then
$\pr\varphi\subset\pr(\psi\circ\varphi)$, which by \eqref{f-dt-1},
gives $(\varphi,\psi)\in\delta_{\Phi}$. This means that
\eqref{f-67} also is satisfied.

\medskip

{\sc Sufficiency}. Let $(G,\cdot,\curlywedge,\xi,\delta)$ be an
algebraic system satisfying all the conditions of the theorem.
Then, by Proposition \ref{P-TVS}, $\xi$ is a reflexive and
symmetric relation, and $\zeta$ is stable in the semigroup
$(G,\cdot)$. Moreover, the implication
\begin{equation} \label{f-dt-trs}
  (g_1,g_2)\in\zeta\wedge g_1\in f_{\xi}(\{x,y\})\longrightarrow
  g_2\in f_{\xi}(\{x,y\})
\end{equation}
is valid for all $g_1,g_2,x,y\in G$. In fact, the premise of
\eqref{f-dt-trs} can be rewritten in the form:
\[
(g_1,g_1)\in\xi\wedge (g_1\curlywedge g_1) e\boxdot e\zeta
g_2e\wedge g_1, g_1e\in f_{\xi}(\{x,y\}).
\]
So, if it is satisfied, then, according to the definition of
$F_{\xi}(H)$ and Lemma \ref{Lem}, $g_2\in
F_{\xi}(f_{\xi}(\{x,y\}))= f_{\xi}(\{x,y\})$, which proves
\eqref{f-dt-trs}.

Now we show that for every $x,y\in G$ the subset $G\setminus
f_{\xi}(\{x,y\})$ is a right ideal of the semigroup $(G,\cdot)$.
Indeed, if $gu\in f_{\xi}(\{x,y\})$, then, by \eqref{f-57}, for
some natural $n$ we have $gu\in\stackrel{n}{F}_{\xi}(\{x,y\})$.
Hence
\[
(gu,gu)\in\xi\wedge (gu\curlywedge gu) e\boxdot e\zeta
gu\wedge gu, gue\in\stackrel{n}{F}_{\xi}(\{x,y\}),
\]
so, $g\in\stackrel{n+1}{F}_{\!\!\xi}\!\!(\{x,y\})\subset f_{\xi}
(\{x,y\})$. Thus, $g\in f_{\xi}(\{x,y\})$. In this way we have
shown the implication $gu\in f_{\xi}(\{x,y\})\longrightarrow g\in
f_{\xi}(\{x,y\})$, which by the contraposition is equivalent to
the implication $g\not\in f_{\xi}(\{x,y\})\longrightarrow
gu\not\in f_{\xi}(\{x,y\})$. The last implication means that
$G\setminus f_{\xi}(\{x,y\})$ is a right ideal.

If $(u,v)\in\xi$ for $u,v\in f_{\xi}(\{x,y\})$, then, obviously,
\[
(u,v)\in\xi\wedge (u\curlywedge v) e\delta e\wedge (u\curlywedge
v) ee\zeta (u\curlywedge v) e\wedge u, ve\in f_{\xi}(\{x,y\}).
\]
Thus $u\curlywedge v\in F_{\xi}(f_{\xi}(\{x,y\}))=
f_{\xi}(\{x,y\})$, since the set $f_{\xi}(\{x,y\})$ is
$f_{\xi}$-closed. So, $f_{\xi}(\{x,y\})$ satisfies the implication
\begin{equation} \label{f-68}
  (u,v)\in\xi\wedge u, v\in f_{\xi}(\{x,y\})\longrightarrow
  u\curlywedge v\in f_{\xi}(\{x,y\}).
\end{equation}

We show now that the relation
\[
\varepsilon_{(g_1,g_2)}=\{(x,y)\,|\,x\curlywedge y\in
f_{\xi}(\{g_1,g_2\})\, \vee\, x,y\not\in f_{\xi}(\{g_1,g_2\})\}
\]
defined on the semigroup $(G,\cdot)$ is a right regular equivalence and
$G\setminus f_{\xi}(\{g_1,g_2\})$ is an equivalence class.

The reflexivity and symmetry of $\varepsilon_{(g_1, g_2)}$ are
obvious. To prove the transitivity let
$(x,y),(y,z)\in\varepsilon_{(g_1,g_2)}$. If $x,y,z\not\in f_{\xi}
(\{g_1,g_2\})$, then clearly $(x,z)\in\varepsilon_{(g_1,g_2)}$. In
the case $x\curlywedge y\in f_{\xi}(\{g_1,g_2\})$ from
$x\curlywedge y\zeta y$, by \eqref{f-dt-trs}, we conclude
$y\in f_{\xi}(\{g_1,g_2\})$. Therefore $x,z\in f_{\xi}(\{g_1,
g_2\})$. Consequently, $x\curlywedge y,\, y\curlywedge z\in
f_{\xi}(\{g_1,g_2\})$. But $(x\curlywedge y)\zeta y$,
$(y\curlywedge z)\zeta y$ and $y\xi y$, hence the last,
by \eqref{f-44}, implies $(x\curlywedge y)\xi (y\curlywedge
z)$. From this, applying \eqref{f-68}, we deduce $x\curlywedge
y\curlywedge z\in f_{\xi}(\{g_1,g_2\})$. On the other hand
$(x\curlywedge y \curlywedge z)\zeta (x\curlywedge z)$ for all
$x,y,z\in G.$ So, $x\curlywedge y\curlywedge z\in
f_{\xi}(\{g_1,g_2\})$, according to \eqref{f-dt-trs}, implies
$x\curlywedge z\in f_{\xi}(\{g_1,g_2\})$. Hence
$(x,z)\in\varepsilon_{(g_1,g_2)}$. This proves the transitivity of
$\varepsilon_{(g_1,g_2)}$. Summarizing  $\varepsilon_{(g_1,g_2)}$
is an equivalence relation.

If $x,y\in G\setminus f_{\xi}(\{g_1,g_2\})$, then
$(x,y)\in\varepsilon_{(g_1,g_2)}$. This means that a subset
$G\setminus f_{\xi}(\{g_1,g_2\})$ is contained in some
$\varepsilon_{(g_1, g_2)}$-class. Now let $x\in G\setminus
f_{\xi}(\{g_1,g_2\})$ and $(x,y)\in\varepsilon_{(g_1,g_2)}$. The
case $x\curlywedge y\in f_{\xi}(\{g_1,g_2\})$ is impossible,
because in this case $x\in f_{\xi}(\{g_1,g_2\})$. So, $y\not\in
f_{\xi}(\{g_1,g_2\})$, i.e., $y\in G\setminus
f_{\xi}(\{g_1,g_2\})$. Hence $G\setminus f_{\xi}(\{g_1,g_2\})$
coincides with some $\varepsilon_{(g_1,g_2)}$-class.

To prove that the relation $\varepsilon_{(g_1,g_2)}$ is right
regular let $(x,y)\in\varepsilon_{(g_1,g_2)}$. If $x,y\in
G\setminus f_{\xi}(\{g_1,g_2\})$, then $xz,yz\in G\setminus
f_{\xi}(\{g_1,g_2\})$ since $G\setminus f_{\xi}(\{g_1,g_2\})$ is a
right ideal. Thus $(xz,yz)\in f_{\xi}(\{g_1,g_2\})$. Now if
$x\curlywedge y\in f_{\xi}(\{g_1,g_2\})$ and $xz\in
f_{\xi}(\{g_1,g_2\})$. Then
\[
(x\curlywedge y,x)\in\xi\wedge (x\curlywedge y) z\delta
e\wedge (x\curlywedge y) ze\zeta (x\curlywedge y) ze\wedge (x
\curlywedge y), xz\in f_{\xi}(\{g_1,g_2\}),
\]
whence, by \eqref{f-54}, we obtain $(x\curlywedge y)z\in
f_{\xi}(\{g_1,g_2\})$. But $(x\curlywedge y) z\zeta yz$, hence
$yz\in f_{\xi}(\{g_1,g_2\})$. Similarly, from $x\curlywedge y\in
f_{\xi}(\{g_1,g_2\})$ and $yz\in f_{\xi}(\{g_1,g_2\})$ we get
$xz\in f_{\xi}(\{g_1,g_2\})$. So, if $x\curlywedge y\in
f_{\xi}(\{g_1,g_2\})$, then $xz$, $yz$ belong or do not belong to
$f_{\xi}(\{g_1,g_2\})$ simultaneously. If $xz,yz\not\in
f_{\xi}(\{g_1,g_2\})$, then obviously,
$(xz,yz)\in\varepsilon_{(g_1,g_2)}$. If $xz,yz\in
f_{\xi}(\{g_1,g_2\})$, then, as was shown above, from
$x\curlywedge y\in f_{\xi}(\{g_1,g_2\})$ it follows $(x\curlywedge
y)z\in f_{\xi}(\{g_1,g_2\})$. Since $(x\curlywedge y)z\zeta
xz$ and $(x\curlywedge y) z\zeta yz$, then obviously
$(x\curlywedge y)z\zeta (xz\curlywedge yz)$. Hence
$xz\curlywedge yz\in f_{\xi}(\{g_1,g_2\})$, i.e.,
$(xz,yz)\in\varepsilon_{(g_1,g_2)}$. So, in any case
$(x,y)\in\varepsilon_{(g_1,g_2)}$ implies
$(xz,yz)\in\varepsilon_{(g_1,g_2)}$. This proves that
$\varepsilon_{(g_1,g_2)}$ is right regular.

From what was just shown, it follows that the pair
$(\varepsilon^*_{(g_1,g_2)},W_{(g_1,g_2)})$, where
$$
\varepsilon^*_{(g_1,g_2)}=\varepsilon_{(g_1,g_2)}\cup\{(e,e)\},\quad\quad
W_{(g_1,g_2)}=G\setminus f_{\xi}(\{g_1,g_2\}),
$$
is a determining pair of the semigroup $(G,\cdot)$.

Let
$\big(P_{(\varepsilon^*_{(g_1,g_2)},W_{(g_1,g_2)})}\big)_{(g_1,g_2)\in
G\times G}$ be the family of simplest representations of the
semigroup $(G,\cdot)$. Their sum
\begin{equation} \label{f-52a}
P=\sum\limits_{(g_1,g_2)\in G\times G}\!\!\!
P_{(\varepsilon^*_{(g_1,g_2)},W_{(g_1,g_2)})}
\end{equation}
is a representation of $(G,\cdot)$ by transformations. It is easy
to see that the above determining  pairs satisfy \eqref{f-6},
\eqref{f-7} and \eqref{f-8}. Therefore, by Proposition \ref{P-p2},
we have
\[
P_{(\varepsilon^*_{(g_1,g_2)},W_{(g_1,g_2)})}(x\curlywedge y)=
P_{(\varepsilon^*_{(g_1,g_2)},W_{(g_1,g_2)})}(x)\cap
P_{(\varepsilon^*_{(g_1,g_2)},W_{(g_1,g_2)})}(y)
\]
for all $g_1,g_2\in G$. Hence $P(x\curlywedge y)=P(x)\cap P(y)$
for $x,y\in G$. Thus, $P$ is the homomorphism of the algebra
$(G,\cdot,\curlywedge)$ onto the $\cap$-semigroup
$(\Phi,\circ,\cap)$, where $\Phi = P(G)$.

Now we prove that $\xi=\xi_{P}$ and $\delta=\delta_{P}$. In fact,
according to \eqref{f-dt-4} and \eqref{f-dt-7} we have
\begin{eqnarray}
& & (x,y)\in\xi_P\;\;\longleftrightarrow\!\!\!
\bigcap\limits_{(g_1, g_2)\in G\times G}\!\!\!
\xi_{(\varepsilon^*_{(g_1,g_2)},W_{(g_1,g_2)})}\;\longleftrightarrow
\nonumber\\
 & & (\forall g_1)(\forall g_2) (\forall u\in G^*)\big(ux,uy\in f_{\xi}(\{g_1,g_2\})
\longrightarrow ux\curlywedge uy\in
f_{\xi}(\{g_1,g_2\})\big).\nonumber
\end{eqnarray}
The last implication for $u=e$ and $g_1=x$, $g_2=y$ has the form
\[
x,y\in f_{\xi}(\{x,y\})\longrightarrow x\curlywedge y\in
f_{\xi}(\{x,y\}).
\]
Thus $x\curlywedge y\in f_{\xi}(\{x,y\})$. Hence, by \eqref{f-66},
we obtain $x\xi y$. This proves $\xi_P\subset\xi$.

To prove the converse inclusion let $(x,y)\in\xi$. If $ux,uy\in
f_{\xi}(\{g_1,g_2\})$ for some $u\in G^*$ and $g_1,g_2\in G$, then
from $(x,y)\in\xi$, by the left regularity of $\xi$, we obtain
$(ux,uy)\in\xi$, which by \eqref{f-68} implies $ux\curlywedge
uy\in f_{\xi}(\{g_1,g_2\})$. Therefore
$(ux,uy)\in\xi_{(\varepsilon^*_{(g_1,g_2)},W_{(g_1,g_2)})}$. Thus
$(x,y)\in\!\!\bigcap\limits_{(g_1,g_2)\in G\times G}\!\!
\xi_{(\varepsilon^*_{(g_1,g_2)},W_{(g_1,g_2)})}=\xi_P$. So,
$\xi\subset\xi_P$. Consequently, $\xi=\xi_P$.

Now if $(x,y)\in\delta$ and $ux\in f_{\xi}(\{g_1,g_2\})$ for some
$g_1,g_2\in G$ and $u\in G^*$, then also $(ux,y)\in\delta$ because
$\delta$ is a left ideal of $(G,\cdot)$. Since
$f_{\xi}(\{g_1,g_2\})$ is $f_{\xi}$-closed, $(ux,y)\in\delta$
together with $ux\in f_{\xi}(\{g_1,g_2\})$, according to
\eqref{f-23}, imply $uxy\in f_{\xi}(\{g_1,g_2\})$. Thus
$(x,y)\in\delta_{(\varepsilon^*_{(g_1,g_2)},W_{(g_1,g_2)})}$.
Hence $(x,y)\in\!\!\bigcap\limits_{(g_1,g_2)\in G\times G}\!\!\!
\delta_{(\varepsilon^*_{(g_1,g_2)},W_{(g_1,g_2)})}= \delta_P$.
This proves $\delta\subset\delta_P$.

Conversely, let $(x,y)\in\delta_P$. Then, in view of
\eqref{f-dt-4} and \eqref{f-dt-8}, we have
\[
  (\forall g_1)\big(\forall g_2) (\forall u\in G^*) (ux\in f_{\xi}(\{g_1,g_2\})\longrightarrow
  uxy\in f_{\xi}(\{g_1,g_2\})\big),
\]
which for $u=e$ and $g_1=g_2=x$ has the form
$$
x\in f_{\xi}(\{x\})\longrightarrow xy\in f_{\xi}(\{x\}).
$$
Thus $xy\in f_{\xi}(\{x\})$. This, by \eqref{f-67}, implies
$(x,y)\in\delta$. So, $\delta_P\subset\delta$, and consequently,
$\delta_P=\delta$.

In this way we have shown that $P$ is a homomorphism of
$(G,\cdot,\curlywedge,\xi,\delta)$ onto the $\cap$-semigroup
$(\Phi,\circ,\cap,\xi_{\Phi},\delta_{\Phi})$, where $\Phi=P(G)$.

It is also an isomorphism. To prove this fact observe first that
$\zeta_P\subset\zeta$. Indeed, according to \eqref{f-dt-4} and
\eqref{f-dt-6}, we have:
\begin{eqnarray}
&&(x,y)\in\zeta_P\;\longleftrightarrow\!\!\!\bigcap\limits_{(g_1,
g_2)\in G\times G}\!\!\!
\zeta_{(\varepsilon^*_{(g_1,g_2)},W_{(g_1,g_2)})})\;\longleftrightarrow
\nonumber\\
&&(\forall g_1) (\forall g_2) (\forall u\in G^*)\big(ux\in f_{\xi}
(g_1,g_2) \longrightarrow ux\curlywedge uy\in
f_{\xi}(\{g_1,g_2\})\big).\nonumber
\end{eqnarray}
Putting in the last implication $u=e$ and $g_1=g_2=x$ we obtain
$$
x\in f_{\xi}(\{x\})\longrightarrow x\curlywedge y\in
f_{\xi}(\{x\}).
$$
So, $x\curlywedge y\in f_{\xi}(\{x\})$. This, by \eqref{f-65},
gives $x\zeta y$, i.e., $(x,y)\in\zeta.$ Hence
$\zeta_P\subset\zeta$.

Now let $P(g_1)=P(g_2)$. Then $P(g_1)\subset P(g_2)$ and
$P(g_2)\subset P(g_1)$. Hence $(g_1,g_2)\in\zeta_P$ and
$(g_2,g_1)\in\zeta_P$. This implies 
$(g_1,g_2),(g_2,g_1)\in\zeta$. Thus $g_1=g_2$ because $\zeta$ is a
semilattice order. So, $P$ is a isomorphism between
$(G,\cdot,\curlywedge,\xi,\delta)$ and
$(\Phi,\circ,\cap,\xi_{\Phi},\delta_{\Phi})$.
\end{proof}

Now, using \eqref{f-57} and the formula $\frak{X}_n(z,H)$ from our
Proposition \ref{P-10ab} we can write the conditions \eqref{f-65},
\eqref{f-66} and \eqref{f-67} in the form of systems of elementary
axioms $(A_n)_{n\in\mathbb{N}}$, $(B_n)_{n\in\mathbb{N}}$ and
$(C_n)_{n\in\mathbb{N}}$, respectively, where
\begin{eqnarray}
&&A_n\colon\,\frak{X}_n(x\curlywedge y,\{x\})\longrightarrow x\curlywedge y = x,\nonumber\\
&&B_n\colon\,\frak{X}_n(x\curlywedge y,\{x,y\})\longrightarrow (x,y)\in\xi,\nonumber\\
&&C_n\colon \,\frak{X}_n(xy,\{x\})\longrightarrow
(x,y)\in\delta.\nonumber
\end{eqnarray}
Thus, we have proved the following theorem:

\begin{theorem}\label{T-7ab}
An algebraic system $(G,\cdot,\curlywedge,\xi,\delta)$, where
$(G,\cdot)$ is a semigroup, $(G,\curlywedge)$ is a semilattice,
$\xi,\delta$ are binary relations on $G$, is isomorphic to some
transformative $\cap$-semigroup of transformations
$(\Phi,\circ,\cap,\xi_{\Phi},\delta_{\Phi})$ if and only if
$\,\xi$ is a left regular relation containing a semilattice order
$\zeta$, $\delta$ is a left ideal relation on $(G,\cdot)$, the conditions
$\eqref{f-43}$, $\eqref{f-44}$, $\eqref{f-45}$, as well as the
axioms systems $(A_n)_{n\in\mathbb{N}}$, $(B_n)_{n\in\mathbb{N}}$
and $(C_n)_{n\in\mathbb{N}}$ are satisfied by all elements of $G$.
\end{theorem}

The relation of semicompatibility and the relation of
semiadjacency in a semigroup of transformations can be
characterized by essentially infinite systems of elementary axioms
(for details see \cite{Schein}, \cite{pavl-2} and
\cite{garv-trokh}). Probably the axioms systems
$(A_n)_{n\in\mathbb{N}}$, $(B_n)_{n\in\mathbb{N}}$,
$(C_n)_{n\in\mathbb{N}}$ are also essentially infinite, i.e., they
are not equivalent to any finite subsystems, but this problem
requires further investigation.

\noindent{\bf Acknowledgment}

The authors are highly grateful to referees and M. V. Volkov for their valuable comments and suggestions for improving
the paper.

\begin{minipage}{60mm}
\begin{flushleft}
Dudek~W. A. \\
 Institute of Mathematics and Computer Science \\
 Wroclaw University of Technology \\
 50-370 Wroclaw \\
 Poland \\
 Email: Wieslaw.Dudek@im.pwr.wroc.pl
\end{flushleft}
\end{minipage}
\hfill
\begin{minipage}{60mm}
\begin{flushleft}
 Trokhimenko~V. S. \\
 Department of Mathematics \\
 Pedagogical University \\
 21100 Vinnitsa \\
 Ukraine \\
 Email: vtrokhim@gmail.com
 \end{flushleft}
 \end{minipage}
\end{document}